\documentclass{proc-l}
\usepackage{amssymb}
\usepackage{amsmath}
\usepackage[english]{babel}
\usepackage{amsfonts,amsthm,amscd}

\theoremstyle{plain}
 \newtheorem{theorem}{Theorem}[section]
 \newtheorem{exmp}[theorem]{Example}
 \newtheorem{Rem}[theorem]{Remark}

\newtheorem{thm-def}[theorem]{Theorem/Definition}
\newtheorem{defi}[theorem]{Definition}
\newtheorem{Prop}[theorem]{Proposition}
\newtheorem{Lem}[theorem]{Lemma}
\newtheorem{Cor}[theorem]{Corollary}

\newcommand{\rk}{{\rm rk\,}}
\newcommand{\ind}{{\rm ind}}
\newcommand{\homnk}{{\rm Hom} (\mathbb C^n, \mathbb C^{n+k})}
\newcommand{\indPHN}{{\rm ind}_{\rm PHN}\,}
\newcommand{\indrad}{{\rm ind}_{\rm rad}\,}

\def\Projan{\mathop{\rm Projan}}

\let\:=\colon

\def\dim{\hbox{\rm dim\hskip 2pt}}

\let\Sum=\sum \def\sum{\Sum\nolimits}

\def\part#1#2{{\partial#1\over\partial#2}}

\begin{document}

\author[T. Gaffney]{Terence Gaffney}\thanks{T.~Gaffney was partially supported by PVE-CNPq Proc. 401565/2014-9}
 \address{T. Gaffney, Department of Mathematics\\
  Northeastern University\\
  Boston, MA 02215}
\email{t.gaffney@neu.edu}
\author[M. A. Ruas]{Maria Aparecida Ruas}\thanks{M.A.~Ruas was partially supported by FAPESP Proc. 2014/00304-2 and CNPq Proc. 
	306306/2015-8  }
\address {M. A. Ruas\\
  Instituto de Ci\^encias Matem\'aticas e de Computa\c{c}\~ao - USP\\
 Av. Trabalhador s\~ao-carlense, 400 - Centro\\
CEP: 13566-590 - S\~ao Carlos - SP, Brazil}
\email{maasruas@icmc.usp.br}

\subjclass[2010]{Primary }

\date{}

\title{ Equisingularity and  EIDS}
\begin{abstract}
We continue the study of the equisingularity of determinantal singularities for essentially isolated singularities (EIDS). These singularities are generic except at isolated points. 
\end{abstract}

\maketitle

\section{Introduction}

In this paper we continue the search for necessary and sufficient invariants for the Whitney equisingularity of a family of spaces. We want invariants which depend only on the members of the family, whose independence of parameter ensures that an equisingularity condition holds.  Successful examples include hypersurfaces (
(\cite{GG}) and complete intersections with isolated singularities (\cite{G-2}). The attempt to extend these results to smoothable determinantal singularities led to the introduction of the notion of the landscape of a singularity. 

Choosing the  landscape of a singularity $X$ consists of defining the allowable deformations that include the set, and its generic perturbations. Each set should have a unique generic element that it deforms to up to equivalence.  Describing the connection between  the infinitesimal geometry of $X$ and the topology of the generic element related to $X$ is part of understanding the landscape. 


In \cite{G-R} there is an example of a singularity for which there are two different choices of landscape. In the example, for each choice of landscape, there is a Whitney equisingular family which contains $X$. The invariants of $X$ which control the Whitney equisingularity of the family depend on the choice of landscape. On the infinitesimal side, the key distinguishing feature is the invariant related to the curvature of the module of allowable infinitesimal first order deformations.  On the topological side, each choice of landscape gives a different generic object to which $X$ deforms, and their differing topology is reflected in the differing infinitesimal invariants. For further discussion of the notion of a landscape see \cite{GS}.

In this paper we extend the framework of \cite{G-R} to the simplest non-smooth\-a\-ble and non-isolated singularities--the essentially isolated determinantal  singularities (EIDS). These were introduced as an object of study by  Ebeling and 
Gusein-Zade in \cite{E-GZ}.

A determinantal singularity $X^d\subset{\mathbb C}^q$ is defined by the minors of a matrix $F_X$ whose entries are elements of ${\mathcal O}_q$, where $d$ the dimension of $X$ is the expected dimension for the order of the chosen minors. Thus, a determinantal singularity $X$, with presentation matrix $F_X$, can be viewed as the intersection of the graph of $F_X$, seen as a map from ${\mathbb C}^q$ to $Hom ({\mathbb C}^n, {\mathbb C}^{n+k})$, with ${\mathbb C}^q\times {\Sigma}^t$, where ${\Sigma}^t$, the matrices in $Hom ({\mathbb C}^n, {\mathbb C}^{n+k})$ of $\text{rank}< t,\,  t\leq n.$   A determinantal singularity is an EIDS, if $F_X$ is transverse to the rank stratification except possibly at the origin. This is discussed in greater detail in the next section.

Here is the framework we extend to this case. The equisingularity conditions of interest  are concerned with limits of tangent hyperplanes; these are studied by a modification of the family of spaces; the failure of the equisingularity condition is equivalent to the fiber of the modification being larger than expected. The modifications are spaces associated with modules. The existence of large dimensional fibers in the modification is equivalent to the non-emptiness of the polar variety of the module. This part of the framework was introduced in \cite{Gaff1}.

The invariant $m_d(X)$ defined below measures the vanishing topology in the smoothing, and is related to the polar varieties of the relative Jacobean module of the family. Its independence from parameter is equivalent to the emptiness of the polar variety of the Jacobean module.  

The invariant  $m_d(X)$, through the theory of the multiplicity of pairs of modules, can be related to invariants related to the infinitesimal geometry of the singularity. We do this  in \ref{inf}. This allows us to compute the invariant in terms of the presentation matrix of $X$ independent of the family of which $X$ is a member.

The information that we need about the topology of the generic fiber in order to give a topological interpretation of $m_d(X)$ is already contained in \cite{E-GZ}. We apply the needed information in \ref{mtop}. 

These two steps, the relation of $m_d(X)$ to the topology of the generic element, and its calculation in terms of infinitesimal invariants were introduced into the framework in \cite{G-R}.

In this paper the sets under consideration have non-isolated singularities in general. 
In the equations that appear before \ref{mtop} the fact that our singularities are non-isolated appears very clearly; the topological invariant of a stratum $X$ depends on the infinitesimal invariants of all the strata in $\bar{X}$. This indicates a phenomenon which we can look for in other non-isolated settings. The statement of \ref{We} is similar, in that we are naturally led to consider all of the strata at once due to their interaction, as opposed to showing that the pairs of strata are Whitney separately. \ref{We}  gives a necessary and sufficient criterion for the Whitney equisingularity of families of essentially isolated determinantal singularities (EIDS).  We give an application of our results to the characterization of  generic hyperplanes using our invariants. We also give an example of our theorems, which is of independent interest for what it shows about the multiplicity of a pair of modules in the determinantal setting.




\section{Definitions and previous results}

We start with some notation. 

Given the germ of an analytic set $X$ let $X_{reg}$ denote the smooth points of $X$. A determinantal variety $X$ (of type $(n+k,n,t)$) in an open domain $U\subset\mathbb C^q$ is a variety of dimension $d:=q-(n-t+1)(n+k-t+1) $ defined by the condition $\rk F(x) < t$ where $t\le n$, $F(x)=\left(f_{ij}(x)\right)$ is an $(n+k)\times n$-matrix ($i=1,\ldots, n+k$, $j=1,\ldots, n$), whose entries $f_{ij}(x)$  are complex analytic functions on $U$. In other words, $X$ is defined by all $t\times t$-minors of the matrix $F(x)$. 

This definition can be reformulated in the following way. We can view  $F$ as a map to $\homnk$, and let $\Sigma^t$ be the subset of $\homnk$ consisting of linear maps of rank less than $t$. The variety $\Sigma^t$ has codimension $(n+k-t+1)(n-t+1)$ in $\homnk$. 

The representation of the variety $\Sigma^{t}$ as the union of $\Sigma^{i}\setminus \Sigma^{i-1}$, $i=1, \ldots, t$, is a stratification of $\Sigma^{t}$, which is locally holomorphically trivial.
The  determinantal variety $ _tX$ is the preimage $F^{-1}(\Sigma^t)$ of the variety $\Sigma^t$ (subject to the condition that ${\rm codim}\, _tX = {\rm codim}\,\Sigma^t$).
For $0\le i\le t$, let $_iX=F^{-1}(\Sigma^{i})$. If $X=_iX$ then we drop $i$ from the notation.

All $X$ have an associated fixed presentation matrix, $F_X$. (We suppress $X$ from the notation if it is understood.) Allowable deformations of $X$ arise from deformations of the entries of $F$ so the format of $F$ also determines the allowable first order infinitesimal deformations.

If $F$ is transverse to the rank stratification of $\homnk$ for $x\ne 0$, then the germ of $X$ at a singular point is holomorphic to either  the product of $\Sigma^t$ with an affine space or a transverse slice of  $\Sigma^t$ . The following definitions were given by Ebeling and Gusein-Zade in  \cite{E-GZ}:

\begin{defi} A point $x \in X=F^{-1}(\Sigma^t)$ is called {\it essentially non-singular} if, at the point $x$, $F$ is transversal to the corresponding stratum of the variety $\Sigma^t$. 
\end{defi}


\begin{defi}
A germ $(X,0)\subset(\mathbb C^q,0)$ of a determinantal variety has an {\em isolated essentially singular point} at the origin if it has only essentially non-singular points in a punctured neighborhood of the origin in $X$. Such a germ is an {\it essentially isolated determinantal singularity} or {EIDS}.
\end{defi}

Any ICIS is an example of an {EIDS}.  If an {EIDS} has an isolated singularity, we call it an {IDS}.

EIDS can be studied from the viewpoint of singularities of maps. The group  $\mathcal G=\mathcal R \ltimes \mathcal H,$  semidirect product of the groups of germs of diffeomorphisms
$\mathcal R= \{ h: (\mathbb C^q,0) \to (\mathbb C^q,0) \}$ and $\mathcal H= \{ (A,B) :  A \in GL_n(\mathcal O_q), B \in GL_m(\mathcal O_q)\},$
where
$GL_n(\mathcal O_q)$ and $GL_m(\mathcal O_q)$ are invertible matrices with entries in $\mathcal O_q$ acts 
on the space of map-germs $F: (\mathbb C^q,0) \to \homnk .$   We say that $F_1\sim_{\mathcal G} F_2$ if there exist $h \in \mathcal R,$ $(A,B) \in \mathcal H$
such that $F_2=Ah^{*}(F_1)B^{-1}.$ 
The group $\mathcal G$ is a geometric
 subgroup of the contact group (see \cite{miriam}). The germs of determinantal singularities at  essentially non-singular points correspond to $\mathcal G$- stable germs;  the
varieties  $_iX=F^{-1}(\Sigma^i) $ are  EIDS, for all $1\leq i \leq \text{min}\{n,m\}$ if and only if $F$ is $\mathcal G$-finitely determined
(\cite{miriam}).  

A set is {\it stable} if the $\mathcal O_q$ module of its allowable infinitesimal deformations is the same as its Jacobean module. Thus, the germs of the 
$\{\Sigma^i\}$ at all points are stable. A set $S$ is {\it universal} if the allowable infinitesimal deformations of $F^{-1}(S)$ are the pullback by $F$ of the allowable  infinitesimal deformations of $S$. Thus,  the germs of the 
$\{\Sigma^i\}$ at all points are universal. When $S$ is a determinantal variety,
the module of its infinitesimal deformations is called the normal module, denoted  by $N(S).$ 

A  deformation $F_s: U \subset  \mathbb C^q   \to \operatorname{Hom}(\mathbb C^n,\mathbb C^{n+k})$ of $F$ which is transverse to the rank stratification is called a \textit{stabilization} of $F$. The variety $ X_s=  F^{-1}_s(\Sigma^t)$ is an
\textit{essential smoothing} of $X$ (\cite[Section 1]{E-GZ}). The topology of the essential smoothings is uniquely given for fixed choice of landscape; they are also called  the {\it singular Milnor fibers} of the EIDS $X.$

Let $\tilde F: U \subset  \mathbb C^q \times \mathbb C   \to \operatorname{Hom}(\mathbb C^n,\mathbb C^{n+k})$ be a one parameter deformation  of $F: (\mathbb C^q,0)\to \operatorname{Hom}(\mathbb C^n,\mathbb C^{n+k}),$  such that $\tilde F_s$ is a stabilization of $F$ for all $s \neq 0$ and hence
 $\widetilde {\mathcal X}_s= \widetilde F_s^{-1}(\Sigma^t)$ an
	\textit{essential smoothing} of $X=F^{-1}(\Sigma^t)$ for all $s \neq 0$.  We call $\tilde F$ a {\it stabilization family}, $\widetilde {\mathcal X}$ an {\it essential smoothing family}. 



\section{Families of EIDS}

We will consider families of {EIDS}; families are denoted by $\mathcal X$. 


If we are studying the Whitney equisingularity of a family we denote the parameter space by $Y$ and assume $Y$ is embedded in $\mathcal X$ as a linear subspace.  For simplicity of exposition we impose the following condition on our families.

\begin{defi} A $s$-parameter family  $\mathcal X$ of {EIDS} is {\it good}  if there exists a neighborhood $U$ of $Y$ such that  $F_{\mathcal X(y)}$ is transverse to the rank stratification off the origin for all $y\in U$. 
\end{defi}

This condition enables us to prove:

\begin{Prop} Suppose  $_t\mathcal X$ is a good family of {EIDS} with representation matrix $F$. Then $\{ _i\mathcal X_{reg}\setminus Y\}$, $i=1,\dots ,t$ is a locally analytically trivial stratification of 
$\mathcal X\setminus Y,$ where $ _i\mathcal X_{reg}$ denotes the stratum $_i\mathcal X\setminus  _{i-1}\mathcal X.$ 
\end{Prop}
\begin{proof} From the definition of good family for $i$,  $\{ _i\mathcal X_{reg}\setminus Y\}$ is smooth, since $F_{\mathcal X}$ is transverse to the rank stratification. That the strata are locally analytically trivial follows because the $\Sigma^i$ are and $F$ is transverse to these.
\end{proof} 

The good condition implies that in a $s$-parameter family, no set  of points from $ _i\mathcal X_{reg}$ can split off from the origin if the expected dimension of $ _i\mathcal X_{reg}$ is less than $s$. It also implies that there cannot be subsets of points where $\{ _i\mathcal X_{reg}\setminus Y\}$ is singular which include the origin in their closure even if the expected dimension of $\{ _i\mathcal X_{reg}\}$ is greater than or equal to $s$.







One of the main invariants we will use is $m_d(X^d)$. This was introduced in \cite{G1} for the study of ICIS singularities, and for isolated singularities whose versal deformations have a smooth base in  \cite{Gaff1}. It was used in \cite{G-R}, \cite{BOT} and \cite{RP} to study a smoothable {IDS}.
In the {IDS} context, it is the number of critical points that a generic linear form has when restricted to the generic fiber in a smoothing $\mathcal X$ of $X$. 

 In the {EIDS} context, instead of a smoothing, we have a stabilization--a determinantal  deformation of $X$ to the generic fiber (the essential smoothing of $X$.) 
 The notion of $d$-polar variety extends to this setting as we now explain.

	Let $\tilde F$ be a stabilization family of $X$, $\widetilde {\mathcal{X}}$ an essential smoothing family of $X$.
	
	We  write
	$$\pi :\widetilde {\mathcal{X}}\subset \mathbb C^{q}\times \mathbb C \to \mathbb C,$$  ${\pi}^{-1}(s)=\widetilde{\mathcal X}_s$, $\widetilde{\mathcal X}_0=X$.
	Let $S(\widetilde {\mathcal{X}})$ be the singular set of $\widetilde{\mathcal X},$
	so that   $(\widetilde{\mathcal X})_{reg}=\widetilde{\mathcal X}-S(\widetilde {\mathcal{X}}).$  Choose a linear form $p\: \mathbb C^q\to \mathbb C$, consider the critical set of the map ${{(\pi,p)} \mid}(\widetilde {\mathcal{X}}-S(\widetilde {\mathcal{X}}))$, and take its closure in $\widetilde {\mathcal{X}}$. This is the relative polar curve of $\widetilde {\mathcal{X}}$ with respect to $p$. If the multiplicity of the polar curve at the origin over $ \mathbb C$ has the generic value among all choices of $p$, then we call the polar curve with respect to $p$, the relative polar curve of $\widetilde {\mathcal{X}}$. Of course, this means that the relative polar curve of $\widetilde {\mathcal{X}}$ has many representatives. 
	
	There is another construction of the polar curve which proves useful. Consider $C(JM_z(\mathcal{X}))$, the conormal modification of $\mathcal{X}$ using the module $JM_z(\mathcal{X})$.  Then,  $C(JM_z(\mathcal{X}))\subset \mathcal{X}\times {\mathbb P}^{q-1}$, and its fiber over a smooth point $x$ of $\mathcal{X}_s$ consists of the tangent hyperplanes to $\mathcal{X}_s$ in ${\mathbb C}^q$ at $x$. Denote the projection of $C(JM_z(\mathcal{X}))$ by $\pi_c$, the projection to ${\mathbb P}^{q-1}$ by $\pi_{{\mathbb P}}$. Then $$P_d(	\widetilde {\mathcal{X}})=\pi_c(\pi^{-1}_{{\mathbb P}}(H)),$$where $H$ is the kernel of a generic projection $p$. (For more details  on these constructions see lecture 3 in \cite{GS}.)
	
\begin{defi}  Let $P_d(	\widetilde {\mathcal{X}})$ be the relative polar curve of an essential smoothing family of $X$, then 
$m_d(X):=m_{\mathbb C}(P_d(\widetilde {\mathcal{X}})),$ where $m_{\mathbb C}$ is the multiplicity of the set $P_d(	\widetilde {\mathcal{X}})$  over ${\mathbb C}$ at the origin.
\end{defi}

From the construction $m_d(X)$ is independent of $p$, and independent of the stabilization, once we fix the landscape.

	  We can ask that $H$, kernel of $p$ is not a limiting tangent hyperplane to $X$ at the origin, and that ${ {p} \mid}_{(\widetilde{\mathcal X}_{s})_{reg}}\to \mathbb C$ is  a Morse function for $s\ne 0$. We verify this last point in the proof of the next proposition.
	






\begin{Prop}\label{prop:morse}
With the above notation
$$ m_d(X)=n_{0}$$ where  $n_{0}$ is the number of critical points of ${{p} \mid}_{(\widetilde{\mathcal X}_{s})_{reg}}$.

\end{Prop}
\begin{proof}
Let $\widetilde{\mathcal X}$ be a family of essential smoothings for $X$. Then $C(JM_z(\widetilde{\mathcal X}))$ has dimension 
$(d+1)+(q-d-1)=q$, as $d+1$ is the dimension of $\widetilde{\mathcal X}$ and $(q-d-1)$ is the fiber dimension of $C(JM_z(\widetilde{\mathcal X}))$ over a generic point. It follows that the dimension of $\pi^{-1}_c(S(\widetilde{\mathcal X}))$ must be at most $q-1$, so we can choose $H$ so that $\pi^{-1}_{{\mathbb P}}(H)$ intersects $\pi^{-1}_c(S(\widetilde{\mathcal X}))$ only at points over the origin, and intersects $C(JM_z(\widetilde{\mathcal X}))$ transversely off $p^{-1}_c(0)$. This implies that the fiber
of $P_d(	\widetilde {\mathcal{X}})$ over a generic point $s$ of ${\mathbb C}$ lies in ${(\widetilde{\mathcal X}_{s})_{reg}}$. We can assume $H$ is not a critical value of the projection of ${{C(JM_z(\widetilde{\mathcal X}))}\mid}_{(\widetilde{\mathcal X})_{reg}}$ to ${\mathbb P}^{q-1}$. This implies that  ${\mathbb C}^q\times H$ and ${{C(JM_z(\widetilde{\mathcal X}))}\mid}_{(\widetilde{\mathcal X}_{s})_{reg}}$ meet transversely for generic $s$. In turn, this means that the points of intersection project to Morse points of  ${{p} \mid}_{(\widetilde{\mathcal X}_{s})_{reg}}$. Since the number of these points is $m_d(X)$, the proposition is proved.
\end{proof}


If  $X=\cup _iX$, for $1\le i\le t$, then it makes sense to consider $m_{d_i}(_iX)$ where $d_i$ is the dimension of $_iX$, and the expected dimension of $_iX$ is greater than $0$. If the expected dimension of $_iX$ is zero, we define $m_0(_iX)$ to be the colength of the ideal defining $_iX$.  As the next proposition shows, $m_0(_iX)$ is the number of points of type $_iX$ on a generic fiber of the stabilization.

\begin{Prop} \label{propo:0dim} Suppose $X$ is a determinantal variety, and $_iX$ has expected dimension $0$. Then $m_0(_iX)$ is the number of points of type $_iX$ on a generic fiber of the stabilization of $X$.\end{Prop}
\begin{proof} In a stabilization $\mathcal X$ of $X$, $_i\mathcal X$ is a determinantal variety, hence Cohen-Macaulay. Consider the projection $\pi$ of $_i\mathcal X$ to the base $Y^1$, where $y$ is the coordinate on the base. Then the degree of $\pi$ restricted to $_i\mathcal X$ is the multiplicity of $(y\circ \pi)$ in the local ring of $_i\mathcal X$ at the origin. Since $_i\mathcal X$ is Cohen-Macaulay, this multiplicity is just the dimension of the ring mod  $(y\circ \pi)$, in turn this is just the colength of the ideal defining $_iX$.
\end{proof}

The invariant $m_d(X^d)$ for an {EIDS} as we have defined it, appears in Ebeling and Gusein-Zade's paper  \cite{E-GZ} as the specialization of Poincar\'{e}-Hopf-Nash index of a differential 1-form $\omega$ denoted $\ind_{\rm PHN} \, \omega$. The invariant  $m_d(X^d)$ is  $\ind_{\rm PHN} \, dl$, $l$ a generic linear form. This follows from {\rm  Proposition 1, \cite{E-GZ} } which shows that the value of this invariant for a generic linear form is the number of Morse points of the restriction of the linear form to the generic fiber of a stabilization, and \ref{prop:morse} shows the same is true for $m_d$. We will use formulas developed for 
$\ind_{\rm PHN} \, \omega$ to connect the values of $m_d(_iX,0)$ with the topology of an essential smoothing of $_iX$.

 
 

{\begin{Prop}\label{mtop} In a family of {EIDS} the $m_{d_i}( \hskip  1pt _iX,0),\, d_i=\dim _iX,$  are constant for all $i$ if and only if   $$(-1)^{d _{i} }{\chi}( _{i}X,0)+(-1)^{d _{i}-1}{\chi}(  _{i}X\cap H,0)$$ are constant for all $i$. 
\end{Prop}
\begin{proof} 

The following formula relating the radial index of a differential form $\omega$ and the $\ind_{\rm PHN} \, \omega$ is due to Ebeling and Gusein-Zade (\cite{E-GZ} proposition 4).
\begin{equation}
\indrad (\omega,X,0) = \sum_{i=1}^t n_{it} \indPHN (\omega,  \hskip 1pt _iX,0) + (-1)^{d -1} \overline{\chi}(X,0),\label{$*$}\end{equation} where $d=\dim X,$ $\overline{\chi}(X,0)$ is the reduced Euler characteristic of the intersection of an essential smoothing of $X$ with a small ball centered at the origin, and the integers $n_{it}$ are given by the formulas 
$n_{it}=
(-1)^{(k)(t-i)} \binom{n-i}{n-t}$.

		Assuming $X= {}_tX$, we can re-write Ebeling-Gusein Zade's formula in \eqref{$*$} using $dl$ for $\omega$ and $m_{d_i}(_iX)$ for $ \indPHN (dl,  \hskip 1pt _iX,0), $ $d_i=\dim _iX,$ apply it to each stratum in turn getting:
	
	$$ (-1)^{d}{\chi}(X,0)+(-1)^{d-1}{\chi}( {X\cap H},0) = \sum_{i=1}^t n_{it} m_{d_i}( \hskip  1pt _iX,0)$$
	
	$$ (-1)^{d _{t-1} }{\chi}( _{t-1}X,0)+(-1)^{d _{t-1}-1}{\chi}({  _{t-1}X\cap H},0) = \sum_{i=1}^{t-1} n_{i(t-1)}  m_{d_i}( \hskip  1pt _iX,0)$$
	$$\dots$$
	$$ (-1)^{d _{1} }{\chi}( _{1}X,0)+(-1)^{d _{1}-1}{\chi}(  {  _{1}X\cap H},0)  = m_{d_1}( \hskip 1pt _1X).$$
	
It is understood that if the last equation corresponds to a stratum of dimension zero, then both sides of the last equation have only $1$ term, since ${\chi}(  {  _{1}X\cap H},0) =0$ and the right hand side is the colength of an ideal as explained above.

\end{proof}

\section{The main result}

 Now we develop the connection between the $0$-dimensional polar multiplicities and infinitesimal invariants of the $_iX$, constructed from modules associated to them.

These invariants are closely related to the theory of integral closure of modules, which we briefly review.

Let  $(X, x)$ be a germ of a complex analytic space and $X$ a
small representative of the germ, and let $\mathcal{O}_{X}$ denote the
structure sheaf on a complex analytic space $X$. For simplicity we assume $X$ is equidimensional, and if $M$ is a sheaf of modules on $X$, $M$  a subsheaf of a free sheaf $F$, then $g$ is the generic rank of $M$ on each component of $X$. 

Suppose $(X, x)$ is the germ of a complex analytic space,
$M$ a submodule of $\mathcal{O}_{X,x}^{p}$. Then $h \in
\mathcal{O}_{X,x}^{p}$ is in the {\it integral closure of $M$}, denoted
$\overline{M}$, if for all analytic $\phi : (\mathbb{C}, 0) \to (X,
x)$, $h \circ \phi \in (\phi^{*}M)\mathcal{O}_{1}$. If $M$ is a
submodule of $N$ and $\overline{M} = \overline{N}$ we say that $M$
is a {\it reduction} of $N$ (see \cite{G-2}).


If a module $M$ has finite colength in $\mathcal{O}_{X,x}^{p}$, it
is possible to attach a number to the module, its Buchsbaum-Rim
multiplicity,  $e(M,\mathcal{O}_{X,x}^{p}).$ We can also define the multiplicity $e(M,N)$ of a pair of
modules $M \subset N$, $M$ of finite colength in $N$, as well, even
if $N$ does not have finite colength in $\mathcal{O}_{X}^{p}$. We refer to \cite{Gaff1}, section 2 for the definition of multiplicity of pairs of modules.

One of our invariants will be the multiplicity of the pair $(JM(_iX), N(_iX))$, where $JM(_iX)$ is the Jacobian module and $N(_iX)$ is the module of infinitesimal first order deformations of $_iX$, induced from the first order deformations  of the presentation matrix of $X$. We must show that $(JM(_iX), N(_iX))$ is well-defined for {EIDS}.

We do this by showing that the transversality of $F$ to the strata of $\homnk$ off the origin implies that off the origin the two modules agree. To do this we use the holomorphic triviality of the stratification of $\homnk$. To show that the multiplicity is well defined, it is actually enough to have a Whitney stratification, so we treat this case as well in the next lemmas.

We first work on the target, then on the source. Recall that if $(S,x)$ is the germ of a submanifold at $x$, then a direct transversal to $S$ at $x$ is a plane $T$ of dimension complementary to $S$ which is transverse to $S$. Given such  a plane $T$, we may make a change of coordinates so that the ambient space is holomorphic to $T\times S$. In this coordinate system we denote by $JM(X)_T$, the submodule of $JM(X)$ generated by the partial derivatives of the defining functions with respect to a basis of $T$.

\begin{Lem} \label{direct}Let $X$ be  a stratified subset of $ ({\mathbb C}^N,x)$.

{\rm i)}  Suppose $X$ has a locally holomorphically trivial stratification. Let $S_x$ denote the stratum containing $x$. Suppose $T$ is a direct transversal to $S$ at $x$. Then $JM(X)\subset {JM(X)_T}$.

{\rm ii)} Suppose the stratification is a Whitney stratification, and  let $S_x$ denote the stratum containing $x$. Suppose $T$ is a direct transversal to $S$ at $x$. Then $JM(X)\subset\overline{JM(X)_T}$.

\end{Lem}

\begin{proof} 

i) Choose local coordinates so that the ambient space is holomorphic to $S_x\times T$, and $X=F^{-1}(0)$. Since $S_x$ is a stratum in a holomorphically trivial stratification, with trivialization of form $({s}, r(s,t))$, $r(s,0)=(0)$, for all $i$, $\part{}{s_i}$ lifts to a holomorphic field $\xi$ tangent to every stratum of $X$, $\xi=\part{}{s_i}+\sum h_j(s,t)\part{}{t_j}$, $h_j(s,t)\in m_S$. This implies that
$DF(\xi)=0$, hence, $\part{F}{s_i}=-\sum h_j(s,t)\part{F}{t_j}$, hence $JM(X)_S\subset m_SJM(X)_T$.

ii) Since $S_x$ is a stratum in a Whitney stratification, it is known (\cite {G-2}, Theorem 2.5) that $JM(X)_S\subset {\overline {m_SJM(X)_T}}$, from which the result follows.

\end{proof}

\begin{Lem}\label {pullback} Suppose $F: ({\mathbb C}^q,0) \to  ({\mathbb C}^N,x)$, $x\in X$, $X\subset( {\mathbb C}^N,x)$.

{\rm i)} Suppose $X$ has a Whitney stratification. Let $S_x$ denote the stratum containing $x$. Suppose $F$ is transverse to $S_x$ at $x$. Let $X_F$ denote $F^{-1}(X)$. Then $F^*(JM(X))\subset\overline{JM(X_F)}$.

{\rm ii)} Suppose $X$ has a locally holomorphically trivial stratification. Let $S_x$ denote the stratum containing $x$. Suppose $F$ is transverse to $S_x$ at $x$. Let $X_F$ denote $F^{-1}(X)$. Then $F^*(JM(X))\subset {JM(X_F)}$.

\end{Lem}

\begin{proof} i) Since $F$ is transverse to $S_x$ at $x$, there is a linear space $T_F$ of $\mathbb C^q$, whose image under $DF$ is a direct transversal to  $S_x$ at $x$, $DF(0)$ injective on $T_F$. We assume coordinates on the target chosen to fit with the previous lemma and coordinates on the source so that $\mathbb C^q$ is the product of $T_F$ and a direct transversal $A$, $DF(0)$ takes $T_F$ 
to $T$ in these coordinates, $z$ denotes coordinates on $T_F$, $y'$ denotes coordinates on $A$. Let $G$ be the map defining $X$, so that $G\circ F$ defines $X_F$. Let $\phi: \mathbb C,0\to X_F,x$ be a curve. By the chain rule $JM(X_F)$ is a submodule of $F^*JM(X)$. The generators of $JM(X_F)$, when composed with $\phi$, by the chain rule and choice of coordinates are
$$\{\part {(G\circ F)}{z_i}\circ\phi=\part G{z_i}\circ F\circ \phi\}\mod m_1\phi^*F^*JM(X)$$
as the other partial derivatives of $G\circ F$ are $0 \mod m_1\phi^*F^*JM(X)$.
Since $\part G{z_i}$ pull back to generators of $\phi^*F^*(JM(X))$ on the curve $F\circ \phi$ by the previous lemma, by Nakayama's lemma the $\{\part {(G\circ F)}{z_i}\circ\phi\}$ generate $\phi^*F^*(JM(X))$  as well. Hence, $F^*(JM(X))\subset\overline{JM(X_F)}$ by the curve  criterion.

The proof of ii) is even easier. We again use the chain rule and local coordinates, this time applying Nakayama's lemma to  $\{\part {(G\circ F)}{z_i}\}$ and $F^*JM(X)$. 

\end{proof}

\begin{Prop}  Let $F\:{\mathbb C}^q\to\homnk$, $X=F^{-1}(\Sigma^i)$, $F$ transverse to $\Sigma^i$ at the origin. Then
$JM(X)=F^*(N(\Sigma^i))=N(X)$.
\end{Prop}
\begin{proof}The stratification of $\homnk$ by rank is holomorphically trivial, so \ref{pullback} implies
$JM(X)=F^*(JM(\Sigma^i))$. Since $N$ is stable and universal, $N(X)=F^*(N(\Sigma^i))=F^*(JM(\Sigma^i))$, and the result follows.

\end{proof}
\begin{Cor} If $(X,x)$ is an {EIDS}, then $e(JM(X), N(X))$ is defined.
\end{Cor}
\begin{proof} Since $X$ is an {EIDS}, $F$ is transverse to $\{\Sigma^j\}$ in a deleted neighborhood of the origin. Hence, in a deleted neighborhood of the origin, $JM(X)=N(X)$ by the previous proposition. Therefore, since the modules agree on a deleted neighborhood, the multiplicity of the pair is well defined (\cite{KT}).
\end{proof}

\begin{Rem} In the event that a set $X$ has only a Whitney stratification, and $F$ is a map which is transverse to the stratification except at the origin, then the previous couple of propositions still show that the multiplicity of the pair is well defined, provided $N$, the module of allowable infinitesimal deformations in our landscape, is universal and stable.
\end{Rem}

The other invariant we need is $F(\mathbb C^q)\cdot \Gamma_d(\Sigma^i)$, where $d$ is the dimension of $F^{-1}(\Sigma^i)$, and  $\Gamma_d(\Sigma^i)$ is the codimension $d$ polar variety of $\Sigma^i$. This invariant is an intersection number. For, $F(\mathbb C^q)\cdot \Gamma_d(\Sigma^i)$ is the intersection of the graph of $F$ in $\mathbb C^q\times\homnk$  with $\mathbb C^q\times  \Gamma_d(\Sigma^i)$, where $\Gamma_d(\Sigma^i)$ has codimension $d$ in $\Sigma^i$. Now we connect our infinitesimal invariants with the polar invariants.




\begin{Prop} \label{inf} If $\mathcal X$ is a one parameter stabilization of an {EIDS} $_iX^d$ then
$$e(JM(X),N(X))+F(\mathbb C^q)\cdot \Gamma_d(\Sigma^i)=m_d(X).$$
\end{Prop}
\begin{proof} Let $\tilde F$ be the map from $\mathbb C\times \mathbb C^q\to \homnk$ which defines $\mathcal X$. We can arrange for $\tilde F(y)$ and $\Gamma_d(\Sigma^i)$ to be transverse for $y\ne 0$. Then $F(\mathbb C^q)\cdot \Gamma_d(\Sigma^i)$ is the number of points in which $\tilde F(y)$ intersects  $\Gamma_d(\Sigma^i)$. In turn, this is the multiplicity over $\mathbb C$ of 
$\Gamma_d(N(\mathcal X))$ at the origin, as $\tilde F^{-1}(\Gamma_d(\Sigma^i))=\Gamma_d(N(\mathcal X))$.
Now we apply the multiplicity polar theorem (\cite{Gaff1}, Theorem 5.1, \cite{Gaff}, Theorem 3.22)  to $JM_z(\mathcal X)$ and $N(\mathcal X)$ we get:

$$e(JM(X), N(X))+mult_C \Gamma_d(N({\mathcal X}))=mult_C  \Gamma_d(JM_z({\mathcal X}))=m_d(X).$$
which gives the result. Notice that we used the universality of $N$ to identify $N(X)$ with $N({\mathcal X})(0)$.
\end{proof}

We can now complete the link between our infinitesimal invariants and our topological invariants. In the following corollary, we adopt the convention that if the expected dimension of a singularity type is $0$ for $X$, we use the colength of the ideal defining the singularity type for our invariant. If the expected dimension is $0$, then $\Gamma_d(\Sigma^i)=\Sigma^i$, and $F(\mathbb C^q)\cdot \Sigma^i$ becomes the colength of the ideal of $i\times i$ minors of $F$. 

\begin{Cor}\label{formula} For an {EIDS} $X$, 
  \begin{multline*}
     (-1)^{d }{\chi}( X,0)+(-1)^{d-1}{\chi}( {X\cap H},0) = \\ \sum_{i=1}^t n_{it} (e(JM(_iX^{d_i}), N(_iX^{d_i})))+F(\mathbb C^q)\cdot \Gamma_{d_i}(\Sigma^i)).
     \end{multline*}

Further, in a family of {EIDS} $\mathcal X$, the invariants $$e(JM(_i{\mathcal X}(y)^{d_i}), N(_i {\mathcal X}(y)^{d_i}))+\tilde F(y)(\mathbb C^q)\cdot \Gamma_{d_i}(\Sigma^i)$$ are independent of $y$ iff the invariants

 $$(-1)^{\dim _{i}{\mathcal X}(y)}{\chi}( _{i}{\mathcal X}(y),0)+(-1)^{dim _{i}{\mathcal X}(y)-1}{\chi}( {{  _{i}X\cap H}}(y),0)$$ are independent of $y$ for all $i$. 

\end{Cor}
\begin{proof} This follows from \ref{inf} and the formulas of \ref{mtop}.\end{proof}

The intersection numbers $F(\mathbb C^q)\cdot \Gamma_d(\Sigma^i)$ depend only on the presentation matrix $F$. In the case of $\Sigma^n$, this number has been computed in terms of the entries of $F$ in \cite{G-R}.


Turning to Whitney equisingularity, we pull together some general facts about this condition for the convenience of the reader.

\begin {Prop}\label{Wfacts}
Suppose $\mathcal X\subset {\mathbb C}^q\times{ \mathbb C}^{\ell}$ is a family of sets over $Y=0\times { \mathbb C}^{\ell}$. Then the following are equivalent:

{\rm 1)} $W$ holds for the pair $\mathcal X_{reg},Y$ at the origin (and hence on a neighborhood of the origin).

{\rm 2)} $JM(X)_Y\subset {\overline {m_YJM(\mathcal X)}}$.

{\rm 3)} $JM(X)_Y\subset {\overline {m_YJM_z(\mathcal X)}}$.

{\rm 4)} The exceptional divisor of the blowup of $C(\mathcal X)$ by $m_Y$ has equidimensional fibers over $Y$.

{\rm 5)} The exceptional divisor of the blowup of $\Projan {\mathcal R} (JM_z({\mathcal X}))$ by $m_Y$ has equidimensional fibers over $Y$. 
\end {Prop}
\begin{proof} The equivalence of 1) and 3) is \cite {G-2}, Theorem 2.5. The equivalence of 2) and 3) is easy; 3) directly implies 2) while 2) coupled with the curve criterion for integral closure and Nakayama's lemma implies 3). The equivalence of 1) and 4) is \cite{T-2}, Theorem 1.2, Chapter 5. The equivalence of 3) and 5) is seen as follows: 3) implies that  ${\overline {m_YJM(\mathcal X)}}={\overline {m_YJM_z(\mathcal X)}}$, which implies that the map from the blowup of $C(\mathcal X)$ by $m_Y$  to  the blowup of $\Projan {\mathcal R} (JM_z({\mathcal X}))$ by $m_Y$ is finite. Since 3) implies 4), 5) holds. If 5) holds, then 3) holds because by \cite{KT1}, Theorem A1,  the equidimensionality of the fibers of the exceptional divisor of the blowup of $\Projan {\mathcal R} JM_z({\mathcal X})$ by $m_Y$ over $Y$ implies there is no obstruction to the integral closure condition of 3) which holds generically over $Y$, as W holds generically, extending over all of $Y$.

\end{proof}

\begin{theorem} \label{We} Suppose ${\mathcal X}^{d+\ell}$ is a good $\ell$-dimensional family of {EIDS}. Then the family is Whitney equisingular if and only if the invariants $$e(m_YJM(_i{\mathcal X}(y)),N(_i{\mathcal X}(y)))+\tilde F(y)(\mathbb C^q)\cdot \Gamma_{d_i}(\Sigma^i)$$ are independent of $y$. 
\end{theorem}
\begin{proof} Since the family is good we only need to control the pairs of strata $(V,Y)$, $V$ some stratum in the canonical stratification of $\mathcal X-Y$. From the good hypothesis we know all  strata (except perhaps $Y$) have the  expected dimension, and the types of dimension zero are controlled by the colength of the defining ideals.  Considering the positive dimensional strata, the  Whitney conditions hold for the open stratum of our singularity, provided the fiber of the blow-up of the conormal modification by $m_Y$ is equidimensional over $Y$ (\ref{Wfacts}). This occurs if and only if the polar variety of codimension $d$ of $m_YJM_z(\mathcal X)$ is empty. We claim the independence of our invariants from $y$ hold if and only if this polar is empty.

Given $F\: {\mathbb C}^q\times {\mathbb C}^{\ell}\to Hom(n,n+k)$ which defines $\mathcal{X}$, let $\tilde F\: {\mathbb C}\times{\mathbb C}^q\times {\mathbb C}^{\ell}\to Hom(n,n+k)$ be a family of stabilizations; which means that $\tilde F_0$ is a stabilization of the germ of $F_0$ at the origin, and $\tilde F_y$ is a stabilization of $F_y,0$
for all $y$ in a Zariski open subset of ${\mathbb C}^{\ell}$. We briefly recall the construction of $\tilde F$ and $\tilde F_0$. Let $G\: {\mathbb C}^q\times {\mathbb C}^{\ell}\times Hom(n,n+k)\to Hom(n,n+k)$ given by 
$$G(z,y,A)=F(z,y)+A, G_0=G(z,0,A).$$
Then $G_0$ is a submersion, so $\{G^{-1}_0(\Sigma^i-\Sigma^{i-1})\}$ is a collection of manifolds. Consider the projections to $Hom(n,n+k)$ given by restricting the projection to $Hom(n,n+k)$ to each element of $\{G^{-1}_0(\Sigma^i-\Sigma^{i-1})\}$. Select $A_0\in Hom(n,n+k)$, such that $A_0$ is not in any discriminant of these projections. Then we may assume the same is true for $tA_0$, $t\ne 0$, $t$ small. By the method of proof of Morlet's lemma, (\cite {M}, p 6-05), the map-germ 
$\tilde F_0(t,z,0)=F(z,0)+tA_0$ is transverse to the rank stratification of $Hom(n,n+k)$ for each fixed, and sufficiently small $t\ne 0$, hence 
$\tilde F_0$ is a stabilization of $F_0$. For the construction of $\tilde F$, work with $\{G^{-1}(\Sigma^i-\Sigma^{i-1})\}$ and consider the projections to ${\mathbb C}^{\ell}\times Hom(n,n+k)$. We can assume the choice of $A_0$ is such that for $t$ sufficiently small, $t\ne 0$, there exists a Zariski open set $U$ of ${\mathbb C}^{\ell}$ such that $(y,tA_0)$ does not lie in any of the discriminants in ${\mathbb C}^{\ell}\times Hom(n,n+k)$ for $y\in U$, $t$ sufficiently small, perhaps depending on $y$, $t\ne 0$. Then, again by the method of proof of Morlet's lemma, the map-germs  $\tilde F_y(t,z)=G(y,z,tA_0)$ are transverse to the rank stratification of $Hom(n,n+k)$ for each fixed, and sufficiently small $t\ne 0$, $y\in U$. 

Now, we apply the multiplicity polar theorem again. 
We have 
$$e( JM({\mathcal X}_0), F^*_0(JM(\Sigma^i)))-e( JM({\mathcal X}_y), F^*_y(JM(\Sigma^i)))=$$
$$deg_Y\Gamma_d(JM_z({\mathcal X}))-deg_Y\Gamma_d( F^*(JM(\Sigma^i))).$$
By a branched cover argument (see \cite{G-R} Theorem 2.10 for details) 
$$deg_Y\Gamma_d( F^*(JM(\Sigma^i)))=\tilde F_0(\mathbb C^q)\cdot \Gamma_d(\Sigma^i)-\tilde F_y(\mathbb C^q)\cdot \Gamma_d(\Sigma^i).$$

Then, 
$$deg_Y\Gamma_d(JM_z({\mathcal X}))=e( JM({\mathcal X}_0), F^*_0(JM(\Sigma^i)))+\tilde F_0(\mathbb C^q)\cdot \Gamma_d(\Sigma^i)$$
$$-(e( JM({\mathcal X}_y), F^*_y(JM(\Sigma^i))))+\tilde F_y(\mathbb C^q)\cdot \Gamma_d(\Sigma^i)).$$

So, $deg_Y\Gamma_d(JM_z({\mathcal X}))$ is $0$ if and only if $e(m_YJM(_i{\mathcal X}(y)),N(_i{\mathcal X}(y)))+\tilde F(y)(\mathbb C^q)\cdot \Gamma_{d_i}(\Sigma^i)$ are independent of $y$.
\end{proof}

\begin{Cor}\label{coro:whitney} Suppose ${\mathcal X}^{d+\ell}$ is a good $\ell$-dimensional family of {EIDS}. Then the family is Whitney equisingular if and only if the polar multiplicities at the origin of ${}_i{\mathcal X}(y)$ and the invariants $$e(JM(_i{\mathcal X}(y)),N(_i{\mathcal X}(y)))+\tilde F(y)(\mathbb C^q)\cdot \Gamma_{d(i)}(\Sigma^i)$$ are independent of $y$. 
\end{Cor}
\begin{proof} There is an expansion formula for $e(m_YJM(_i{\mathcal X}(y)),N(_i{\mathcal X}(y)))$ as a sum of terms with the polar multiplicities of $_i{\mathcal X}(y)$ with combinatorial coefficients, and the term  $e(JM(_i{\mathcal X}(y)),N(_i{\mathcal X}(y)))$ (\cite{KT} Theorem 9.8 (i) p.221).
 So Whitney implies 
the polar multiplicities are constant, and $$e(m_YJM(_i{\mathcal X}(y)),N(_i{\mathcal X}(y))+\tilde F(y)(\mathbb C^q)\cdot \Gamma_{d_i}(\Sigma^i)$$  are constant. So the terms $$e(JM(_i{\mathcal X}(y)),N(_i{\mathcal X}(y))+\tilde F(y)(\mathbb C^q)\cdot \Gamma_{d_i}(\Sigma^i)$$ are independent of $y$ as well. The other direction is clear.
\end{proof}



\subsection{Conormal of stably isolated singularities}

We say an {EIDS} is {\it stably isolated} if its essential smoothing has only isolated singularities. If $X=F^{-1}({\Sigma}^t)$, $F\:{\mathbb C}^q\to\homnk$ then $X$ is stably isolated if its essential smoothing $X(y)$ contains isolated points in $_{t-1}X(y)$. This can happen only if $q=$ codim $\Sigma^{t-1}$. These singularities are useful test cases for invariants which are defined for isolated singularities. A ``good" invariant should be zero at a point of $X$, if the dimension of the fiber of  $C(X),$ the conormal of $X,$ is less than or equal to the dimension of the conormal less two. We give a formula for the dimension of the fiber of the conormal of the essential smoothing of a stably isolated singularity at a singular point, showing that this bound always holds. This implies that stably isolated singularities are deficient conormal singularities in the sense of \cite{Ra}, so the machinery of \cite{Ra} applies. 

\begin{Prop} Let $F\:{\mathbb C}^q\to\homnk$, $X=F^{-1}({\Sigma}^t)$, $F$ transverse to $\Sigma^{t-1}$ at the origin, $F(0)\in\Sigma^{t-1}$, $0$ an isolated point of  $F^{-1}(\Sigma^{t-1}\setminus \Sigma^{t-2})$, then the difference between the dimension of $C(X)$ and  the dimension of the fiber of $C(X)$ at $0$ is $k+1$.
\end{Prop}

\begin{proof} The assumptions imply that $F({\mathbb C}^q)$ is a direct transversal to $\Sigma^{t-1}$, so we can assume coordinates chosen as in \ref{direct}, with a local trivialization. The existence of the trivialization implies that ${\Sigma}^t$ and $C({\Sigma}^{t})$ are locally trivial along $\Sigma^{t-1}\setminus \Sigma^{t-2}$ and the fiber of $C({\Sigma}^{t})$ over $F(0)$ in this trivialization is determined by the tangent hyperplanes to the fiber of  ${\Sigma}^t$ over $F(0)$. By choice of trivialization this fiber is isomorphic to $X$. The local trivialization also implies that every tangent hyperplane to ${\Sigma}^t$  at points on $X$ contains the tangent space $T$ to $\Sigma^{t-1}$ at $F(0)$. This gives a bijective map between $C(X)$ and the fiber of $C({\Sigma}^{t})$ over $F(0)$, by taking a point $(x,H)$ of  $C(X)$ and mapping it to $(x, H \oplus T)$. 
	
	So, it suffices to compute the dimension of the fiber of  $p\:C({\Sigma}^{t})\to {\Sigma}^{t}$ over $F(0)$. This will be the dimension of the fiber of $C(X)$ over $0$. In \cite{G-R}, this fiber was computed. We describe the result.
	
	Given $M\in \homnk$, let $K(M)$ denote the kernel of $M$ and $Co(M)$ the cokernel. Let $\Sigma_r(M)$ denote the elements of $Hom(K^*(M),Co(M))$ of kernel rank $r$. Denote $\mathbb P(\overline{\Sigma}_r(M))$ by $X_r(M)$. We suppose $M$ is in  $\Sigma_s$, $s> r$; then the result of \cite{G-R} is that the fiber of the conormal of $C(\overline{\Sigma}_r)$ at $M$ is $X_{s-r}(M)$.
	
	In the case at hand we take $r=n-t+1,$ $s-r=1$ and $M=F(0)$, so the dimension of the fiber of $C(X)$ over $0$ is dim$X_{1}(F(0))$.
	Since dim$K(F(0))=r+1$ and dim$Co(F(0))=k+r+1$, the codimension of $\Sigma_1(F(0))$ in $Hom(K^*(F(0)),Co(F(0)))$  is $k+1$.
	
	Now we have :
	$$\dim C(X)-\dim C(X)(0)=(q-1)-\dim X_{1}(F(0))$$
	$$=q-\dim \Sigma_1(F(0)).$$
	Since $q= \dim Hom(K(F(0)),C(F(0)))$, we have 
	$$\dim C(X)-\dim C(X)(0)=k+1.$$

\end{proof}
\section{Generic hyperplane sections}
As an application, we give a numerical condition for a hyperplane to be a generic hyperplane for $X$, relating this condition to the topology of the sections of an {EIDS}. 

We say a hyperplane is a {\it generic } hyperplane for $X,x$ if it is not a limit of tangent hyperplanes  to the smooth part of $X$ at $x$. A hyperplane $H$ is stratified generic for $X,x$, if $X,x$ has a stratification with strata $\{S_i\}$ and $H$ is generic for all $S_i$. In general this notion depends on the stratification;  adding additional strata results in fewer planes being generic. However, for an {EIDS} $X$, we have a canonical stratification of $X$ given by the $_iX_{reg}$, so when we talk about $H$ being stratified generic for $X$, it is always with this stratification in mind.




\begin{Prop} Given an {EIDS} $X$, and a hyperplane $H$, $H$ is generic for $X$, if and only if 
$$e(JM(X)_H, N(X))=e(JM(X), N(X)),$$and $H$ is stratified generic if $$e(JM( _iX)_H, N(_iX))=e(JM( _iX), N( _iX))$$ for all $i$ such that $_iX$ has dimension at least $1$.\end{Prop}
\begin{proof} We have the multiplicity of pairs of modules is additive (Cf. \cite{KT}); hence for all $i$,
$$e(JM( _iX)_H, N( _iX))=e(JM(_iX), N(_iX))+e(JM(_iX)_H, JM(_iX)).$$

Then $$e(JM(_iX)_H, N(_iX))=e(JM( _iX), N( _iX))$$if and only if
$e(JM( _iX)_H, JM( _iX))=0$. Since all of the $_iX$ are equidimensional, this holds if and only if $JM(_iX)_H$ is a reduction 
of $JM(_iX)$ for all $i$ (\cite {KT}). Then, by  \cite{G0},  $JM(X)_H$ is a reduction of $JM(X)$ if and only if $H$ is generic for $X$, and  $JM(_iX)_H$ is a reduction 
of $JM(_iX)$ for all $i$ if and only if $H$ is stratified generic. \end{proof}

The additivity result also shows that  for any hyperplane $H$ for which the multiplicities are defined, $e(JM(_iX)_H, N(_iX))\ge e(JM(_iX), N(_iX))$, so that the stratified generic hyperplanes minimize these multiplicities. 


Next, we show that the stratified generic hyperplanes give generic plane sections for all of the formulas of type \ref{formula}, establishing a minimality result in terms of the topology of these sections. We say that a hyperplane $H$ is {\it topologically minimal} for an {EIDS} $X$, if for all $i$ such that $_iX$ has positive dimension, $(-1)^{dim _iX-1}{\chi}( {_iX\cap H},0)$ is minimal among all hyperplanes $H$.

\begin{Prop} Let $X$ be an {EIDS} singularity. Then $H$ is stratified generic for $X$ if and only if $H$ is topologically minimal for $X$.
\end{Prop}
\begin{proof} It follows from \ref{formula} and the last proposition that  $H$ is stratified generic for $X$ if and only if 
  \begin{multline*}
     (-1)^{\dim _jX }{\chi}(_jX,0)+(-1)^{dim _jX-1}{\chi}( {_jX\cap H},0) = \\ \sum_{i=1}^{t_j} n_{it_j} (e(JM(_iX^{d_i})_H, N(_iX^{d_i}))+F(\mathbb C^q)\cdot \Gamma_{d(i)}(\Sigma^i)).$$
   \end{multline*}
   
Further, if $H$ is not stratified generic, then the right sides are greater than or equal to the left hand side, and for at least one formula the inequality is strict. Since the term $(-1)^{\dim _jX }{\chi}(_jX,0)$ appears in the formulae independently of $H$, the terms $(-1)^{dim _jX-1}{\chi}( {_jX\cap H},0) $ must be minimal for $H$ stratified generic.
\end{proof}

This result recovers and extends the results of \cite{BCR}.

\begin{Prop}
	
Let $X^d$  be an EIDS singularity  $\mathbb C^q.$   Then 

\begin{itemize}
\item [\rm{(i)}] For every stratified generic hyperplane $H,$ $X \cap H$ is a $(d-1)$-di\-mension\-al EIDS of the same type in $\mathbb C^{q-1}.$
\item [\rm {(ii)}] The family of sections $$ \{ X \cap H: \,  H  \, \text{is a stratified generic hyperplane for}\,  X  \}$$  is Whitney equisingular.

\end{itemize}  
\end{Prop}	

\begin{proof}
	(i) is Theorem 4.2 in \cite{BCR}
	
	(ii) follows from the above Proposition and Corollary \ref{coro:whitney}.
\end{proof}

\begin{exmp}
	
	Let $X=F^{-1}(\Sigma^2)$, where
	$$\begin{array}{cccl}
	F\:& \mathbb C^{6}      & \rightarrow & Hom(\mathbb C^2, \mathbb C^3) \\
	& (x,y)   & \mapsto     &\left(
	\begin{array}{ccc}
	x_1& x_2 & x_3 \\
	x_4 & x_5   & x_1+y^{k+1} \\
	\end{array}
	\right) 
	
	\end{array}
	,$$	These are corank one determinantal varieties of  type $(3,2,2).$ From Proposition 11.5 in \cite{DP}
	it follows that ${\chi}(X,0)=-\mu(y^{k+1})=-k.$ These singularities also appear in Theorem 3.6 of \cite{F-N}. They are listed as type $\Omega_k$.
		
	The relevant equations from  Proposition \ref{mtop} are:
	
$$	\chi(X,0) - \chi(X\cap H,0)=-m_0(_1X,0)+ m_4(X,0)$$  and $$m_0(_1X,0)=\chi(_1X,0)$$

The stratum $(_1X,0)$ is a $0$-dimensional ICIS of type $A_k,$	so that  $$ m_0(_1X,0)=\mu(g)+1, \ \mu(g)=k.$$
Then, we get    $m_4(X,0)= -\mu(g)+1 + \mu(g)+1 - \chi(X\cap H),$  so $$m_4(X,0)=2- \chi(X\cap H).$$
	
We  can now use the calculations in \cite{FK-Z}, to  find that $\chi(X\cap H)=2$. Specifically, $X\cap H$ is of the same type as the first singularity in the  table in section 5.
	Thus, $m_4(X,0)=0$.
\end{exmp}

In the above example, the equation $m_4(X,0)=0$ can also be deduced from the integral closure machinery. We have that  $ F(\mathbb C^6)\cdot\Gamma_4(\Sigma ^2) =0,$ since  the polar variety of $\Sigma^2$ of codimension $4$ is empty because it would have dimension $0$. 

Further we can see that ${\overline{JM(X)}}={\overline{N(X)}}$ for this singularity, as follows. Let $G\:Hom(2,3)\to \mathbb C^3$ be the map whose components are the maximal minors, so $\Sigma^2=G^{-1}(0)$.  Apply the chain rule to find $D(G\circ F)$. Then $DG(F)$ is the matrix of generators of $N(X)$ since $N$ is stable. If we write  the generators of $JM(X)$ in terms of the generators of $N(X)$, the coefficients come from $DF$, by the chain rule. Consider the ideal sheaf induced on $\Projan{\mathcal R}(N(X))$ by the module with matrix $DF$. Denote it ${\mathcal F}$. Now $\Projan{\mathcal R}(N(X))\subset X\times {\mathbb P}^5$, use coordinates $T_{i,j}$ for $ {\mathbb P}^5$, $1\le i\le 2, 1\le j\le 3$. Then in degree $1$,  ${\mathcal F}$ is generated by $\{ T_{1,1}+T_{2,3}, T_{i,j}, y^{k}T_{2,3}\}$, where $(i,j)\ne (1,1)$ or $(2,3)$. To see this note for example, that $\part F {x_1}$ is 
$$\left(
	\begin{array}{ccc}
	1&0 & 0\\
	0 &0  &1 \\
	\end{array}
	\right) 
	. $$ The corresponding  generator of ${\mathcal F}$ is given by taking the dot product of this matrix with a $2$ by $3$ matrix whose entries are the $T_ {i,j}$. Then $T_{1,1}+T_{3,2}$ is the corresponding element in degree $1$.

We have $V({\mathcal F})$ on $\mathbb C^{6}\times  {\mathbb P}^5$ consists of points of the form $$((x_1,x_2,x_3,x_4,x_5,0),[-1,0,\dots,0,1]).$$ On the other hand, the fiber of $\Projan{\mathcal R}(N(X))$ over $0$ in $X$ consists of the projectivisation of the matrices of rank $1$, (\cite{G-R}) so $V({\mathcal F})$ misses this fiber, since every matrix in $V({\mathcal F})$ has rank $2$ . This implies that ${\overline{JM(X)}}={\overline{N(X)}}$ 
so $e(JM(X),N(X))=0$. Hence, $m_4(X)=0$, which checks the computation from the topological side. 

This example shows that if $m_d(X^d)=0$, the ideal induced from the derivative of the presentation matrix must not vanish on the fiber of $\Projan{\mathcal R}(N(X))$ over $0$.

The first step to compute the multiplicity of the pair $(JM(X),N(X))$ using intersection theory is to write the generators of $JM(X)$ in terms of the generators of $N(X)$ (\cite{KT}). As we have shown above, for EIDS, the chain rule provides a natural way to do this, showing how $DF$ enters into the computation of $e(JM(X),N(X))$. Since $DF$ plays an important role in the Mather-Damon theory of sections of varieties (\cite{D1} p. 95), this is a promising connection between the two theories.

\end{document}